\documentclass[11pt]{article}
\usepackage{amsmath, amsthm, amstext}
\usepackage{amssymb, array, amsfonts}
\usepackage[utf8]{inputenc}
\usepackage[english]{babel}
\usepackage{enumerate}
\usepackage{graphicx}
\usepackage{mathrsfs}

\usepackage{pdfpages}

\inputencoding{utf8}

\newtheorem{thm}{Theorem}
\newtheorem{lemma}{Lemma}

\title{Asymptotic properties of solutions to a certain ultrahyperbolic equation}
\author{M. N. Demchenko\footnote{St.~Petersburg Department of
V.\,A.~Steklov Institute of Mathematics of
the Russian Academy of Sciences, 
27 Fontanka, St.~Petersburg, Russia. E-mail: demchenko@pdmi.ras.ru.\newline
\indent The research was supported by Russian Foundation for Basic Research, grant No. 20-01-00627-a.}}

\date{}

\begin{document}
\maketitle
\begin{abstract}
  We consider a certain ultrahyperbolic equation in a Euclidean space being a generalization of Klein-Gordon-Fock equation.
  The behavior of solutions at points tending to infinity along timelike directions is studied.
  We examine the issue of existence of solutions possessing given asymptotic properties at infinity.

\smallskip


\end{abstract}

\section{Introduction}
We consider the following ultrahyperbolic equation
\begin{equation}
  (\Delta_t - \Delta_x + m^2) u = f,
  \label{eqn}
\end{equation}
where a solution $u(x,t)$ and a function $f(x,t)$ are defined in
${\mathbb R}^d\times{\mathbb R}^n$, $d,n \geqslant 1$, 
\begin{equation*}
  \Delta_x = \partial_{x_1}^2 + \ldots + \partial_{x_d}^2, \quad \Delta_t = \partial_{t_1}^2 + \ldots + \partial_{t_n}^2,
\end{equation*}
$m$ is a positive constant.
In the particular case $n=1$, $f=0$, this is (hyperbolic) Klein-Gordon-Fock equation
describing free motion of a relativistic spinless particle with rest mass $m$.

In the present paper, we consider solutions 
$u(x,t)$ with the following asymptotic property 
\begin{equation}
  u(s \theta, s \omega)
  = s^{-(d+n-1)/2} \sum_\pm U_\pm(\theta,\omega) e^{\pm i s\, m\sqrt{1-\theta^2}}
  + O(s^{-(d+n+1)/2}), \quad s\to+\infty.
  \label{ampl}
\end{equation}
Here $(\theta,\omega)\in B^d\times S^{n-1}$, 
$B^d = \{\theta \in{\mathbb R}^d\,|\, |\theta|<1\}$, $S^{n-1} = \{\omega \in{\mathbb R}^n\,|\, |\omega|=1\}$, and $\theta^2 := |\theta|^2$.
Relation~(\ref{ampl}) characterizes the behavior of a solution at the infinity along timelike directions.
We will find a family of solutions, which exhibit such behavior. 
Besides, for given functions $U_\pm$ (more precisely, either $U_+$ or $U_-$ will be given) and for a given function $f$, 
we will construct a solution to equation~(\ref{eqn}) satisfying~(\ref{ampl}).

Asymptotic properties of solutions to the wave equation (the latter being the particular case of~(\ref{eqn}) for $n=1$, $m=0$)
and its vector analogues were studied in a number of papers (e.g., in~\cite{Lax-Phillips, Blag-scatter, Moses, Kis, BV, Plachenov}).
However, we point out the following considerable difference between the results obtained there and the asymptotic formula~(\ref{ampl}) being
valid in the case $m>0$.
Consider, for example, the case $d=3$, $n=1$.
The solution to the Cauchy problem for the wave equation with compactly supported initial data and $f$ vanishes when $(x,t)$ tends to infinity along
a timelike direction (Huygens principle).
On the other hand, in the case when $(x,t)$ tends to infinity along a {\em characteristic} direction,
the solution exhibits nontrivial asymptotic behavior, which can be described in terms of the Radon transform of the initial data
and a certain integral transform of the function $f$.
Besides, the solution is uniquely determined by the corresponding asymptotic coefficients~\cite{Blag-scatter}.
We will show that in the case of Klein-Gordon-Fock equation (when $m>0$) the situation is quite the reverse:
the solution to the Cauchy problem decay faster than any power of time and distance 
along characteristic directions
and possesses the asymptotic property~(\ref{ampl}) along timelike directions.
This will be established for a certain family of solutions in the general case $n\geqslant 1$
avoiding analysis of the Cauchy problem, which becomes ill-posed.

The set of known well-posed problems for ultrahyperbolic equations 
is less than
that for elliptic, parabolic, and hyperbolic equations.
Paper~\cite{Blag-char} concerns the characteristic problem for equation of the form~(\ref{eqn}) with $m=0$, $f=0$,
in which the solution $u(x,t)$ is determined in the region $|x|<|t|$ (or $|x|>|t|$) from its values on the characteristic cone $|x|=|t|$.
There existence and uniqueness of the solution in a certain class of functions were proved for arbitrary $d,n\geqslant 1$, 
and an integral formula for $u$ was obtained.
In the present paper, we consider the problem for equation~(\ref{eqn}) with $m>0$, in which
one of the coefficients $U_\pm$ from relation~(\ref{ampl}) is given as the data.
We establish the existence of the solution to this problem. 

The author is thankful to A.~P.~Kiselev for 
pointing to related results concerning the wave equation, 
and to M.~I.~Belishev and A.~F.~Vakulenko for helpful discussions.

\section{Construction of solutions possessing property~(\ref{ampl})}
We use the following definition of Fourier transform 
\begin{gather*}
  \hat v(\xi) = \int_{{\mathbb R}^d} e^{-i\langle{}x,\xi\rangle} v(x)  dx
\end{gather*}
(here $\langle\cdot,\cdot\rangle$ is the standard inner product in a real Euclidean space;
further the angle brackets will also denote the pairing of a distribution with a test function).
For a function $v(x,t)$ defined in ``spacetime'' ${\mathbb R}^d\times{\mathbb R}^n$,
it is convenient to define Fourier transform as follows
\begin{gather*}
  \hat v(\xi,\tau) = \int_{{\mathbb R}^{d+n}} e^{i(-\langle{}x,\xi\rangle+\langle{}t, \tau\rangle)} v(x,t)  dx dt.
\end{gather*}
Then inverse Fourier transform has the following form
\begin{gather*}
    v(x,t) = (2\pi)^{-d-n} \int_{{\mathbb R}^{d+n}} e^{i(\langle{}x,\xi\rangle-\langle{}t, \tau\rangle)} \hat v(\xi,\tau) d\xi d\tau.
\end{gather*}

We will always assume that the function $f$ on the right-hand side of equation~(\ref{eqn}) belongs to Schwartz space $\mathcal{S}({\mathbb R}^{d+n})$.
By applying Fourier transform to equation~(\ref{eqn}) we get
\begin{equation*}
  (\xi^2 + m^2 - \tau^2) \hat u(\xi,\tau) = \hat f(\xi,\tau).
\end{equation*}
Then we may formally represent $\hat u$ as follows
\begin{gather}
  \hat u = \hat u^f + \hat u^a,  \label{hatu}  \\
  \hat u^f(\xi,\tau) = \frac{\hat f(\xi,\tau)}{\xi^2 + m^2 - \tau^2}, \quad
  \hat u^a(\xi,\tau) = \delta(\xi^2 + m^2 - \tau^2) a(\xi,\tau), \notag
\end{gather}
which is essentially a sum of a particular solution to the inhomogeneous equation ($u^f$)
and a general solution to the homogeneous equation ($u^a$ with an arbitrary density $a$).
Later we will give rigorous definitions of the functions $u^f$ and $u^a$.
In particular, the factor $(\xi^2 + m^2 - \tau^2)^{-1}$,
which is singular at the hypersurface
\begin{equation*}
  \Sigma_m = \{(\xi,\tau)\in{\mathbb R}^{d+n}\,|\, \xi^2+m^2-\tau^2=0\},
\end{equation*}
will be regularized by means of the Cauchy principal value. 

Now we formally apply inverse Fourier transform to~(\ref{hatu}):
\begin{gather}
  u = u^f + u^a, \label{uint}\\
  u^f(x,t) = (2\pi)^{-d-n} \, \int_{{\mathbb R}^{d+n}} e^{i(\langle{}x,\xi\rangle-\langle{}t, \tau\rangle)} \frac{\hat f(\xi,\tau)}{\xi^2 + m^2 - \tau^2} \, d\xi d\tau, \notag\\
  u^a(x,t) = (2\pi)^{-d-n} \int_{{\mathbb R}^{d+n}} \, e^{i(\langle{}x,\xi\rangle-\langle{}t,\tau\rangle)} \delta(\xi^2+m^2-\tau^2) a(\xi,\tau) \, d\xi d\tau. \notag
\end{gather}
By additional formal transformations, one can bring the expression for $u^f$ to the form
\begin{gather}
  u^f(x,t) = (2\pi)^{-d-n} \,\mathrm{v.p.}\int_{{\mathbb R}^d\times S^{n-1}\times (0,\infty)}
  \frac{e^{i(\langle{}x,\xi\rangle-\langle{}t, \sigma\rangle \rho\sqrt{\xi^2 + m^2})} F(\xi,\sigma,\rho)}{1-\rho} \, d\xi dS_\sigma d\rho,
  \label{uf}\\
  F(\xi,\sigma,\rho) = 
  (\xi^2 + m^2)^{n/2-1} \hat f(\xi,\rho\sigma \sqrt{\xi^2 + m^2}) \frac{\rho^{n-1}}{1 + \rho}, \notag
\end{gather}
where we have chosen a certain regularization of the singularity $(1-\rho)^{-1}$ arising from the singularity $(\xi^2 + m^2 - \tau^2)^{-1}$
in the original expression.
Namely, the integral in~(\ref{uf}) is understood in the sense of Cauchy principal value.
Here and further $dS$ is the measure on the unit sphere $S^{n-1}$ induced by the Euclidean metrics in ${\mathbb R}^n$;
for the sphere of dimension zero, we set
\begin{equation*}
  dS_\sigma = \delta(\sigma-1) + \delta(\sigma+1).
\end{equation*}
Expression~(\ref{uf}), which is taken for the definition of the function $u^f$, will be discussed in sec.~\ref{sec-vp} in more details.
In particular, we will show that $u^f$ is a smooth solution to equation~(\ref{eqn}).
It is possible to chose a regularization of the integral in~(\ref{uf}) that differs from Cauchy principal value.
A solution we would have thus obtained differs from the chosen one by a solution of the homogeneous equation of the form~(\ref{eqn}).

Next the expression for the term $u^a$ in~(\ref{uint}) can be formally transformed to
\begin{equation}
  u^a(x,t) = (2\pi)^{-d-n} \int_{{\mathbb R}^d\times S^{n-1}} e^{i \left(\langle{}x,\xi\rangle - \langle{}t, \sigma\rangle \sqrt{\xi^2+m^2}\right)} A(\xi,\sigma) \, d\xi dS_\sigma,
  \label{uaint}
\end{equation}
where the function $A(\xi,\sigma)$ is related to $a(\xi,\tau)$ as follows
\begin{equation*}
  A(\xi,\sigma) = \frac{1}{2} (\xi^2+m^2)^{n/2-1} a\left(\xi,\sigma\sqrt{\xi^2+m^2}\right).
\end{equation*}
The function $A$ is defined in ${\mathbb R}^d\times S^{n-1}$.
This function, and so $u^a$, depends on the values of $a$ only on $\Sigma_m$.
Define $\mathcal{S}(\Sigma_m)$ as the set of functions $a$ on the hypersurface $\Sigma_m$ that correspond to functions $A$ from Schwartz space
$\mathcal{S}({\mathbb R}^d\times S^{n-1})$.
When this condition is fulfilled, we assume relation~(\ref{uaint}) as a definition of the function $u^a(x,t)$.
As can be easily verified, it is a smooth solution of homogeneous equation of the form~(\ref{eqn}).

The assumptions on the coefficients $U_\pm$ in~(\ref{ampl}) will involve the class of functions in $B^d\times S^{n-1}$
that have ``zero of infinite order'' on the boundary $\partial B^d$.
Namely, we say that $U$ belongs to $\mathcal{S}(B^d\times S^{n-1})$ if
it coincides 
on $B^d\times S^{n-1}$
with a function from $C^{\infty}({\mathbb R}^d\times S^{n-1})$ supported in $\overline{B^d}\times S^{n-1}$.

\begin{thm}\label{thm-asymp}
  Let $f\in\mathcal{S}({\mathbb R}^{d+n})$, $a\in\mathcal{S}(\Sigma_m)$.
  Then the function $u(x,t)$ given by~(\ref{uint}) is a smooth solution of equation~(\ref{eqn}), and asymptotic formula~(\ref{ampl}) holds true. 
  The coefficients $U_\pm$ are expressed in terms of $f$, $a$ as follows
\begin{equation}
  U_\pm(\theta,\omega) =
  \frac{e^{\pm i\pi (d-n+1)/4}}{4\pi m} \left(\frac{m}{2\pi \sqrt{1-\theta^2}}\right)^{(d+n-1)/2}\,
  (a \mp i\pi \hat f)\left(\frac{\mp m\theta}{\sqrt{1-\theta^2}}, \frac{\mp m\omega}{\sqrt{1-\theta^2}}\right)
  \label{Upm}
\end{equation}
(note that for any $(\theta,\omega)\in B^d\times S^{n-1}$, the arguments of the function $a \mp i\pi \hat f$ on the right-hand side
correspond to a point on $\Sigma_m$).
\end{thm}

Now we indicate necessary conditions on $U_\pm$, which follow from~(\ref{Upm}).
In the case $f=0$, we have
\begin{equation}
  U_\pm(-\theta,-\omega) = (\pm i)^{d-n+1} U_\mp(\theta,\omega),
  \label{apm}
\end{equation}
whereas in the case $a=0$, we have
\begin{equation}
  U_\pm(-\theta,-\omega) = (\pm i)^{d-n-1} U_\mp(\theta,\omega).    
  \label{apm-f}
\end{equation}
Thus the coefficients $U_\pm$ in asymptotic formula~(\ref{ampl}) can not be arbitrary functions. 
At the same time, the following theorem holds true.
\begin{thm}\label{thm-exist}
  Let $f\in\mathcal{S}({\mathbb R}^{d+n})$, $U_+\in\mathcal{S}(B^d\times S^{n-1})$.
  Then the function
\begin{equation*}
  a(\xi,\tau) = 4\pi m\, e^{-i\pi(d-n+1)/4} \left(\frac{2\pi}{|\tau|}\right)^{(d+n-1)/2} U_+\left(\frac{-\xi}{|\tau|}, \frac{-\tau}{|\tau|}\right)
  + i \pi \hat f(\xi, \tau)
\end{equation*}
of variables $(\xi,\tau)\in\Sigma_m$
belongs to $\mathcal{S}(\Sigma_m)$, and the corresponding solution $u(x,t)$ of the form~(\ref{uint}) exhibits the asymptotic behavior~(\ref{ampl})
with the given function $U_+$ and a function $U_-\in\mathcal{S}(B^d\times S^{n-1})$ determined by relation~(\ref{Upm}).
Moreover, this is the unique solution of the form~(\ref{uint}) with density $a$ from $\mathcal{S}(\Sigma_m)$,
for which the coefficient $U_+$ in~(\ref{ampl}) coincides with the given function.
\end{thm}
In Theorem~\ref{thm-exist},
the functions $U_+$ and $U_-$ are interchangeable,
i.e., the solution $u$ can be constructed from a given function $U_-$ (and $f$).
In this case, the corresponding expression for $a$ takes the form
\begin{equation*}
  a(\xi,\tau) = 4\pi m\, e^{i\pi(d-n+1)/4} \left(\frac{2\pi}{|\tau|}\right)^{(d+n-1)/2} U_-\left(\frac{\xi}{|\tau|}, \frac{\tau}{|\tau|}\right)
  - i \pi \hat f(\xi, \tau).
\end{equation*}

We point out that there is a number of results concerning fundamental solutions to ultrahyperbolic equations.
We mention only the pioneering paper~\cite{deRham}, where equation~(\ref{eqn}) with $m=0$ was considered,
and paper~\cite{Ortner-Wagner} concerned with the case of arbitrary $m$.
One can use fundamental solutions to construct particular solutions to the inhomogeneous equation~(\ref{eqn}).
However, for the study of asymptotic properties of solutions, the Fourier analysis and the method of stationary phase
used in the present paper
seem to be the most appropriate tools.

\section{Auxiliary assertions}\label{sec-vp}
For a function $F\in\mathcal{S}({\mathbb R}^n)$, we give the following definition of the integral in the sense of Cauchy principal value 
\begin{equation}
  \mathrm{v.p.} \int_{{\mathbb R}^n} \frac{1}{x_n} F(x) dx = \lim_{\varepsilon\to 0+} \int_{|x_n|>\varepsilon} \frac{1}{x_n} F(x)\, dx.
  \label{vpint-general}
\end{equation}

We will need some estimates related to definition~(\ref{vpint-general}) involving
the following func\-ti\-onals on functions in ${\mathbb R}^n$ ($k,p$ are nonnegative integers):
\begin{equation}
  |F|_{k,p} = \sup_{|\alpha|\leqslant k,\, x\in{\mathbb R}^n} (1+|x|)^p |\partial^\alpha F(x)|.
  \label{Fkp}
\end{equation}

For $0<\varepsilon'\leqslant\varepsilon$, we have
\begin{equation*}
  \int_{|x_n|>\varepsilon} - \int_{|x_n|>\varepsilon'} \frac{1}{x_n} F(x) \,dx
  =\int_{{\mathbb R}^{n-1}} dx' \int_{\varepsilon'<|x_n|<\varepsilon} \frac{1}{x_n} F(x) \,dx_n,
\end{equation*}
where we use the notation $x'=(x_1,\ldots,x_{n-1})$.
Now show that
\begin{equation}
  \left|\int_{\varepsilon'<|x_n|<\varepsilon} \frac{1}{x_n} F(x) \,dx_n\right| \leqslant 2 (\varepsilon-\varepsilon') (1+|x'|)^{-n} |F|_{1,n}.
  \label{e'e}
\end{equation}
Indeed, the left-hand side equals
\begin{equation*}
  \left|\int_{\varepsilon'<|x_n|<\varepsilon} \frac{1}{x_n} (F(x)-F(x',0)) \,dx_n\right|
  \leqslant 2(\varepsilon-\varepsilon') \sup_{x_n} |\partial_{x_n} F(x',x_n)|,
\end{equation*}
which implies~(\ref{e'e}).
Now the absolute value of the integral with respect to $x'$ occurring in the preceding relation can be estimated by
$C_n (\varepsilon-\varepsilon') |F|_{1,n}$.
This implies the existence of the limit on the right-hand side of~(\ref{vpint-general}) and the estimate
\begin{equation}
  \left| \int_{|x_n|>\varepsilon} \frac{1}{x_n} F(x)\, dx - \mathrm{v.p.} \int_{{\mathbb R}^n} \frac{1}{x_n} F(x) dx\right| \leqslant C_n \varepsilon |F|_{1,n}.
  \label{vpint-converge}
\end{equation}
Next for $0<\varepsilon\leqslant 1$ we have
\begin{equation*}
  \left|\int_{|x_n|>\varepsilon} \frac{1}{x_n} F(x) \,dx_n\right| \leqslant
  \left|\int_{|x_n|>1} \frac{1}{x_n} F(x) \,dx_n\right|
  + \left|\int_{\varepsilon<|x_n|<1} \frac{1}{x_n} F(x) \,dx_n\right|.
\end{equation*}
The first term is estimated by
\begin{equation*}
  C \sup_{x_n} ((1+|x_n|) \, |F(x',x_n)|) \leqslant C_n (1+|x'|)^{-n} |F|_{0,n+1},
\end{equation*}
while the second one can be estimated using~(\ref{e'e}).
After integrating with respect to $x'$ we obtain
\begin{equation}
  \left|\int_{|x_n|>\varepsilon} \frac{1}{x_n} F(x) \,dx\right| \leqslant C_n |F|_{1,n+1}.
  \label{F11}
\end{equation}
Our calculations also imply the equality
\begin{equation}
  \mathrm{v.p.} \int_{{\mathbb R}^n} \frac{1}{x_n} F(x) \,dx = \int_{{\mathbb R}^{n-1}} dx' \,\mathrm{v.p.}\int_{\mathbb R} \frac{1}{x_n} F(x) \,dx_n.
  \label{vp-foubini}
\end{equation}

Suppose that the function $F$ in the integral~(\ref{vpint-general}) depends on the parameter $\lambda$.
For a fixed $\varepsilon>0$ we have
\begin{equation*}
  \partial_\lambda \int_{|x_n|>\varepsilon} \frac{1}{x_n} F(x,\lambda) \,dx = \int_{|x_n|>\varepsilon} \frac{1}{x_n} \partial_\lambda F(x,\lambda) \,dx,
\end{equation*}
as soon as functionals $|F(\cdot,\lambda)|_{0,n+1}$, $|\partial_\lambda F(\cdot,\lambda)|_{0,n+1}$ are uniformly bounded. 
This follows from the fact that the specified assumption implies that the integrals in this relation exist and
the improper integral on the right-hand side converges uniformly.
Assume also that $|\partial_\lambda F(\cdot,\lambda)|_{1,n}$ is uniformly bounded.
Then due to~(\ref{vpint-converge}) the integral on the right-hand side of the last equality,
and so the left-hand side, tends to its limit as $\varepsilon\to 0$ uniformly with respect to $\lambda$.
This implies that
\begin{multline*}
  \lim_{\varepsilon\to 0+} \int_{|x_n|>\varepsilon} \frac{1}{x_n} \partial_\lambda F(x,\lambda)\, dx
  = \lim_{\varepsilon\to 0+} \partial_\lambda \int_{|x_n|>\varepsilon} \frac{1}{x_n} F(x,\lambda)\, dx \\
  = \partial_\lambda \left(\lim_{\varepsilon\to 0+} \int_{|x_n|>\varepsilon} \frac{1}{x_n} F(x,\lambda)\, dx\right),
\end{multline*}
whence
\begin{equation*}
  \partial_\lambda \left(\mathrm{v.p.} \int_{{\mathbb R}^n} \frac{1}{x_n} F(x,\lambda)\, dx\right) = \mathrm{v.p.} \int_{{\mathbb R}^n} \frac{1}{x_n} \partial_\lambda F(x,\lambda)\, dx.
\end{equation*}
The analogous relation for higher order derivatives with respect to $\lambda$ is also valid: 
\begin{equation}
  \partial_\lambda^\alpha \left(\mathrm{v.p.} \int_{{\mathbb R}^n} \frac{1}{x_n} F(x,\lambda)\, dx\right) = \mathrm{v.p.} \int_{{\mathbb R}^n} \frac{1}{x_n} \partial_\lambda^\alpha F(x,\lambda)\, dx,
  \label{ddvp}
\end{equation}
providing that for all $\alpha$ we have $|\partial_\lambda^\alpha F(\cdot,\lambda)|_{1,n+1} \leqslant C_\alpha$.

Now show that the integral of the form~(\ref{vpint-general}) allows changing of variable. 
We will consider only the case when $n=1$, and the function $F$ belongs to $C_0^\infty(I)$, where
$I\subset{\mathbb R}$ is an open interval containing zero.
Suppose that $z(x)$ is a smooth function on $\overline I$ such that
\begin{equation}
  z(0) = 0, \quad z'|_{\overline I} \ne 0.
  \label{z-monotone}
\end{equation}
By $x(z)$ we denote the mapping inverse to $z(x)$.
We will establish the equality
\begin{equation}
  \mathrm{v.p.}\int_{\mathbb R} \frac{1}{x} F(x)\, dx = \mathrm{v.p.}\int_{\mathbb R} \frac{1}{z} E(z)\, dz, 
  \quad E(z) = \frac{z}{x(z)} |x'(z)|\, F(x(z)).
  \label{vp-change-var}
\end{equation}
The function $E(z)$ is defined by the specified expression on the interval $z(I)$ and is continued by zero on its complement.
Note that due to conditions~(\ref{z-monotone}) the function $E(z)$ is regular at zero, and so it belongs to $C_0^\infty({\mathbb R})$. 

Suppose that $z' > 0$ (the opposite case $z'<0$ is treated analogously).
For small enough $\varepsilon>0$ we have
\begin{equation}
  \int_{|x|>\varepsilon} \frac{1}{x} F(x)\, dx
  = \int_{z(\varepsilon)}^\infty + \int_{-\infty}^{z(-\varepsilon)} \frac{1}{z} E(z) \, dz.
  \label{change-var}
\end{equation}
Set $h := z'(0)$. Then
\begin{equation*}
  \left|\int_{z(\varepsilon)}^\infty - \int_{h \varepsilon}^\infty \frac{1}{z} E(z)\, dz\right| 
  \leqslant
  \frac{|z(\varepsilon) - h\varepsilon| \,\sup |E|}{\min(z(\varepsilon), h\varepsilon)}
  \leqslant C \varepsilon \sup |E|.
\end{equation*}
The integral over the set $z<z(-\varepsilon)$ is analogously approximated by the integral over the set $z<-h\varepsilon$.
Then in view of~(\ref{change-var}) we obtain that
\begin{equation*}
  \left|\int_{|x|>\varepsilon} \frac{1}{x} F(x)\, dx
  -\int_{|z|>h\varepsilon} \frac{1}{z} E(z) \, dz \right| \leqslant C \varepsilon \sup |E|.
\end{equation*}
Sending $\varepsilon$ to zero in this inequality gives~(\ref{vp-change-var}).

Now turn to expression~(\ref{uf}) for the function $u^f(x,t)$.
The form of this expression differs from~(\ref{vpint-general}), since the set of integration is not a Euclidean space.
However, the definition of the integral in the sense of Cauchy principal value is adapted to this case in an obvious way.
The function $F$ in~(\ref{vpint-general}) should be substituted by the function
\begin{equation}
  e^{i(\langle{}x,\xi\rangle-\langle{}t, \sigma\rangle \rho\sqrt{\xi^2 + m^2})} F(\xi,\sigma,\rho)
  \label{func}
\end{equation}
defined on the set $(\xi,\sigma,\rho) \in {\mathbb R}^d\times S^{n-1}\times (0,\infty)$ and depending on the parameters $x$, $t$.
We say that a function defined on the specified set belongs to Schwartz space if
the corresponding values of all of the functionals $|\cdot|_{k,p}$ are finite, the latter being defined (analogously to~(\ref{Fkp})) as follows.
Chose an arbitrary finite atlas on the sphere $S^{n-1}$.
Denote by $\gamma$ a chart from this atlas, i.e., a diffeomorphism of an open subset of the sphere (which will be called the domain of definition
of a chart) onto an open subset of ${\mathbb R}^{n-1}$
($\gamma$ will also denote points from this subset of ${\mathbb R}^{n-1}$, i.e., coordinates).
To each chart $\gamma$, we associate a smooth function $\chi_\gamma(\sigma)$ on the sphere supported in the domain of definition of $\gamma$,
such that
\begin{equation*}
  \sum_\gamma \chi_\gamma \equiv 1.
\end{equation*}
Set
\begin{equation*}
  |F|_{k,p} = 
  \sup (1+|\xi| + \rho)^p \, |\partial^\alpha_\xi \partial^{\alpha'}_\gamma \partial^{\alpha''}_\rho (\chi_\gamma(\sigma(\gamma)) F(\xi,\sigma(\gamma)\,\rho))|,
\end{equation*}
where the supremum is taken over all charts $\gamma$ of the atlas,
points $\xi, \gamma, \rho$, and multi-indices $\alpha, \alpha', \alpha''$,
such that $|\alpha|+|\alpha'|+\alpha''\leqslant k$.

Note that the function~(\ref{func}) belongs to Schwartz space, as soon as so does $\hat f(\xi,\tau)$ (see~(\ref{uf})).
Moreover, the functionals $|\cdot|_{k,p}$ of this function and its derivatives with respect to $(x,t)$ are locally
uniformly bounded, 
which allows for applying formula~(\ref{ddvp}).
We have
\begin{multline*}
  (2\pi)^{d+n} (\Delta_t - \Delta_x + m^2) u^f(x,t)= \\
  = \mathrm{v.p.}\int_{{\mathbb R}^d\times S^{n-1}\times (0,\infty)}
  \frac{(\Delta_t - \Delta_x + m^2)\left(e^{i(\langle{}x,\xi\rangle-\langle{}t, \sigma\rangle \rho\sqrt{\xi^2 + m^2})}\right) F(\xi,\sigma,\rho)}{1-\rho} \, d\xi dS_\sigma d\rho\\
  = \mathrm{v.p.}\int_{{\mathbb R}^d\times S^{n-1}\times (0,\infty)}
  \frac{(1-\rho^2)(\xi^2+m^2) e^{i(\langle{}x,\xi\rangle-\langle{}t, \sigma\rangle \rho\sqrt{\xi^2 + m^2})} F(\xi,\sigma,\rho)}{1-\rho} \, d\xi dS_\sigma d\rho\\
  = \int_{{\mathbb R}^d\times S^{n-1}\times (0,\infty)}
  (\xi^2+m^2)^{n/2} e^{i(\langle{}x,\xi\rangle-\langle{}t, \sigma\rangle \rho\sqrt{\xi^2 + m^2})} \hat f(\xi,\sigma\rho\sqrt{\xi^2+m^2}) \rho^{n-1} \, d\xi dS_\sigma d\rho\\  
  = \int_{{\mathbb R}^d\times{\mathbb R}^n}
  (\xi^2+m^2)^{n/2} e^{i(\langle{}x,\xi\rangle-\langle{}t, \tau\rangle\sqrt{\xi^2 + m^2})} \hat f(\xi,\tau\sqrt{\xi^2+m^2}) \, d\xi d\tau\\  
  = \int_{{\mathbb R}^d\times{\mathbb R}^n}
  e^{i(\langle{}x,\xi\rangle-\langle{}t, \tau\rangle)} \hat f(\xi,\tau) \, d\xi d\tau = (2\pi)^{d+n} f(x,t).
\end{multline*}
Thus the function $u^f$ is a solution to equation~(\ref{eqn}).

\section{The method of stationary phase}
In this section, we consider the behavior of the integral
\begin{equation}
  \mathrm{v.p.} \int_{{\mathbb R}^n} \frac{1}{x_n} F(x) e^{i s \Phi(x)}\, dx
  \label{vpint}
\end{equation}
as $s\to +\infty$, assuming that 
\begin{equation}
  F\in \mathcal{S}({\mathbb R}^n), \quad \Phi\in C^\infty({\mathbb R}^n; {\mathbb R}), \quad |\partial^\alpha\Phi(x)| \leqslant C_\alpha (1+|x|)^{M_{|\alpha|}} \quad \forall\alpha,
  \label{FPhi}
\end{equation}
for some nonnegative integers $M_k$, where $k\geqslant 0$.
First we give the following elementary lemma.
\begin{lemma}\label{stat-triv}
  Assume conditions~(\ref{FPhi}) are satisfied.
  Suppose also that for all $x\in{\rm supp} F$ we have
\begin{equation}
  |\partial\Phi(x)| \geqslant C(1+|x|)^{-M}.
  \label{ddPhi}
\end{equation}
Then for any $N$ we have
\begin{equation*}
  \int_{{\mathbb R}^n} F(x) e^{i s \Phi(x)} dx = O(s^{-N}), \quad s\to \pm\infty.
\end{equation*}
\end{lemma}
\begin{proof}
  By integrating by parts, we get that
\begin{multline*}
  \int_{{\mathbb R}^n} F(x) e^{i s \Phi(x)} dx 
  = -i s^{-1} \int_{{\mathbb R}^n} \frac{F(x)}{|\partial\Phi(x)|^{2}} \langle\partial\Phi(x),\partial\rangle e^{i s \Phi(x)}\, dx \\
  = i s^{-1} \int_{{\mathbb R}^n} (\langle\partial\Phi(x),\partial\rangle + \Delta\Phi(x))\left(\frac{F(x)}{|\partial\Phi(x)|^{2}}\right) e^{i s \Phi(x)}\, dx.
\end{multline*}
In view of the assumptions on the functions $F$ and $\Phi$, the exponential under the integral sign is multiplied
by a function from Schwartz space in ${\mathbb R}^n$ supported on a set contained in ${\rm supp} F$.
The latter means that estimate~(\ref{ddPhi}) holds true on the support of this function,
and so we can apply our argument once more.
After sufficient number of iterations, we can obtain the factor $s^{-N}$ for any given $N$.
\end{proof}

Further, as in sec.~\ref{sec-vp}, we will use the notation $x' = (x_1,\ldots,x_{n-1})$.
\begin{lemma}\label{stat-decay}
  Assume conditions~(\ref{FPhi}) are satisfied.
  Suppose also that for all $x\in{\rm supp} F$ we have
\begin{equation}
  |\partial_{x'}\Phi(x)| \geqslant C(1+|x|)^{-M}.
  \label{ddx'}
\end{equation}
Then the integral~(\ref{vpint}) equals $O(s^{-N})$ as $s\to \pm\infty$ for any $N$.
\end{lemma}
\begin{proof}
  We will estimate the integral
\begin{equation}
  \int_{|x_n|>\varepsilon} \frac{1}{x_n} F(x) e^{i s \Phi(x)}\, dx
  \label{int-eps}
\end{equation}
for an arbitrary $\varepsilon>0$.
This integral equals
\begin{multline*}
  -i s^{-1} \int_{|x_n|>\varepsilon} \frac{F(x)}{x_n |\partial_{x'}\Phi(x)|^{2}} \langle\partial_{x'}\Phi(x),\partial_{x'}\rangle e^{i s \Phi(x)}\, dx \\
  = i s^{-1} \int_{|x_n|>\varepsilon} \frac{1}{x_n} (\langle\partial_{x'}\Phi(x),\partial_{x'}\rangle + \Delta_{x'}\Phi(x))\left(\frac{F(x)}{|\partial_{x'}\Phi(x)|^{2}}\right) e^{i s \Phi(x)}\, dx.
\end{multline*}
In view of the assumptions on the functions $F$ and $\Phi$, we obtain the integral of the form~(\ref{int-eps}),
in which the function $F$ is substituted by a certain function $F_1$ from Schwartz space in ${\mathbb R}^n$ supported
on a set contained in ${\rm supp} F$.
The latter means that estimate~(\ref{ddx'}) for the functions $F_1$ and $\Phi$ holds true, and so
we can iterate our argument.
After a sufficient number of iterations, we get that the integral~(\ref{int-eps}) equals
\begin{equation*}
  s^{-N} \int_{|x_n|>\varepsilon} \frac{1}{x_n} F_N(x) e^{i s \Phi(x)}\, dx,
\end{equation*}
where $F_N\in \mathcal{S}({\mathbb R}^n)$.
Then due to~(\ref{F11}), we have
\begin{multline*}
  \left|\lim_{\varepsilon\to 0} \int_{|x_n|>\varepsilon} \frac{1}{x_n} F(x) e^{i s \Phi(x)}\, dx\right|
  = |s|^{-N} \left|\lim_{\varepsilon\to 0} \int_{|x_n|>\varepsilon} \frac{1}{x_n} F_N(x) e^{i s \Phi(x)}\, dx\right| \\  
  \leqslant C |s|^{-N} |F_N e^{i s \Phi}|_{1,n+1}.
\end{multline*}
It remains to observe that
\begin{equation*}
  |F_N e^{i s \Phi}|_{1,n+1} \leqslant |s|\, |F_N|_{1,n+1+M_0+M_1},
\end{equation*}
where $M_0$, $M_1$ are the exponents in the last inequality in hypothesis~(\ref{FPhi}).
\end{proof}

\begin{lemma}\label{lemma-z-asymp}
  Let $F\in\mathcal{S}({\mathbb R})$. Then for any $N\geqslant 0$ and $s>0$ we have
  \begin{equation}
    \left|\mathrm{v.p.} \int_{\mathbb R} \frac{1}{z} F(z) e^{i s z}\, dz - i\pi F(0)\right| \leqslant C_N |F|_{N+1,2}\, s^{-N}.
    \label{z-asymp}
  \end{equation}
\end{lemma}
\begin{proof}
  The integral on the left-hand side of~(\ref{z-asymp}) equals the pairing of the distribution $\mathcal{P}_{1/z}$
  with the test function $F(z) e^{i s z}$ of the variable $z$:
\begin{equation*}
  \mathrm{v.p.} \int_{\mathbb R} \frac{1}{z} F(z) e^{i s z}\, dz
  = \left\langle\mathcal{P}_{1/z}, F(z) e^{i s z}\right\rangle_z.
\end{equation*}
This expression equals the pairing of the distribution $(2\pi)^{-1}\widehat{\mathcal{P}_{1/z}}$
on the inverse Fourier transform of the specified test function.
The latter (as a function of the variable $\xi$) equals $\check F(\xi+s)$.
Therefore, by applying formula
\begin{equation*}
  \widehat{\mathcal{P}_{1/z}}(\xi) = -i\pi\,{\rm sgn}\xi,
\end{equation*}
we get
\begin{multline*}
  \left\langle\mathcal{P}_{1/z}, F(z) e^{i s z}\right\rangle_z
  = -i\pi (2\pi)^{-1} \int_{\mathbb R} \check F(-\xi-s)\, {\rm sgn}\xi\, d\xi \\
  = i/2 \int_{\mathbb R} \check F(\xi) d\xi - i \int_0^\infty \check F(-s-\xi)\, d\xi 
  = i\pi F(0) - i \int_{-\infty}^{-s} \check F(\xi)\, d\xi. 
\end{multline*}
The second term decays rapidly when $s$ grows.
This term can easily be estimated by the right-hand side of~(\ref{z-asymp}).
\end{proof}

\begin{lemma}\label{stat-asymp}
  Assume conditions~(\ref{FPhi}) are satisfied, and, besides, the function $F$ is com\-pact\-ly supported.
  Suppose also that $\Phi$ satisfies the following conditions:
\begin{gather}
  \partial\Phi(x) \ne 0, \quad x\in V, \label{ddxne0}\\
  \partial_{x'}\Phi(x',0) \ne 0, \quad  (x',0)\in V\setminus\{0\}, \label{ddx'ne0} \\  
  \partial_{x'}\Phi(0) = 0, \quad \det\partial_{x'}^2 \Phi(0) \ne 0, \label{ddx'0}
\end{gather}
where $V$ is a neighborhood of ${\rm supp} F\cup\{0\}$. 
Then we have (${\rm sgn}(\partial_{x'}^2 \Phi(0))$ is the difference of number of positive eigenvalues and negative eigenvalues
of the matrix $\partial_{x'}^2 \Phi(0)$)
\begin{multline}
  \mathrm{v.p.} \int_{{\mathbb R}^n} \frac{1}{x_n} F(x) e^{i s \Phi(x)}\, dx =\\
  =i\pi\, {\rm sgn}(\partial_{x_n}\Phi(0)) \left(\frac{2\pi}{s}\right)^{(n-1)/2} |\det \partial_{x'}^2 \Phi(0)|^{-1/2} e^{i\pi\, {\rm sgn}(\partial_{x'}^2 \Phi(0))/4} e^{i s \Phi(0)} F(0)\\
  + O(s^{(n+1)/2}), \quad s \to +\infty.
  \label{stat-phase2}
\end{multline}
\end{lemma}
\begin{proof}
  Condition~(\ref{ddxne0}) and the first condition in~(\ref{ddx'0}) imply that $\partial_{x_n} \Phi(0) \ne 0$.
  Let $\chi(x)$ be a smooth function, such that $\partial_{x_n} \Phi \ne 0$ on its support. 
  Assume also that $\chi(x)=1$ for small $x$.
  The left-hand side in~(\ref{stat-phase2}) can be represented as the following sum
\begin{equation}
  \mathrm{v.p.} \int_{{\mathbb R}^n} \frac{1}{x_n} (\chi F e^{i s \Phi})(x)\, dx
  + \mathrm{v.p.} \int_{{\mathbb R}^n} \frac{1}{x_n} ((1-\chi) F e^{i s \Phi})(x)\, dx.
  \label{vpint-sum}
\end{equation}
Now show that the second term equals $O(s^{-N})$ as $s\to+\infty$ for any $N$.
Decompose it as a sum of two terms:
\begin{equation*}
  \mathrm{v.p.} \int_{{\mathbb R}^n} \frac{1}{x_n} \zeta(x_n) ((1-\chi) F e^{i s \Phi})(x)\, dx
  + \int_{{\mathbb R}^n} \frac{1}{x_n} (1-\zeta(x_n)) ((1-\chi) F e^{i s \Phi})(x)\, dx,
\end{equation*}
where $\zeta$ is a function from $C_0^\infty({\mathbb R})$ that equals unity in the neighborhood of zero.
In view of condition~(\ref{ddxne0}) and the compactness of the support of $F$,
the second term can be treated with the use of Lemma~\ref{stat-triv}.
The function $F(x)$ in Lemma~\ref{stat-triv} should be substituted by the expression
\begin{equation*}
  (1-\zeta(x_n)) (1-\chi(x)) F(x)/x_n.
\end{equation*}
Due to condition~(\ref{ddx'ne0}), the first term in the preceding expression can be treated by Lemma~\ref{stat-decay},
provided that $\zeta$ is supported in a sufficiently small neighborhood of zero.

Now turn to the first term in~(\ref{vpint-sum}).
By~(\ref{vp-foubini}) it can be written in the form
\begin{equation}
  \int_{{\mathbb R}^{n-1}} dx' \,\mathrm{v.p.}\int_{\mathbb R} \frac{1}{x_n} (\chi F e^{i s \Phi})(x)\, dx_n.
  \label{x'xn}
\end{equation}
The condition $\partial_{x_n}\Phi\ne 0$, which is satisfied on the support of the integrand, 
allows for changing the variable $x_n \mapsto z=\Phi(x',x_n)-\Phi(x',0)$ (the inverse mapping will be denoted by $x_n(x',z))$
in the inner integral by formula~(\ref{vp-change-var}), which yields
\begin{equation*}
  e^{i s \Phi(x',0)} \,\mathrm{v.p.}\int_{\mathbb R} \frac{1}{z} \frac{z}{x_n(x',z)} |\partial_z x_n(x',z)|\, \chi(x) F(x) e^{i s z} \, dz.
\end{equation*}
Now apply Lemma~\ref{lemma-z-asymp} to the integral in this expression.
Since $z/x_n|_{z=0} = \partial_{x_n}\Phi(x',0)$, we have
\begin{equation*}
  \frac{z}{x_n(x',z)} |\partial_z x_n(x',z)|\, \chi(x) F(x)\big|_{z=0} = (\chi F\, {\rm sgn}\, \partial_{x_n}\Phi)(x',0) 
  = (\chi F)(x',0)\, {\rm sgn}\, \partial_{x_n}\Phi(0). 
\end{equation*}
We obtain that the integral~(\ref{x'xn}) equals 
\begin{equation*}
  i\pi \,{\rm sgn}\, \partial_{x_n}\Phi(0) \int_{{\mathbb R}^{n-1}} (\chi F e^{i s \Phi})(x',0)\, dx' + O(s^{-N}),
\end{equation*}
as $s\to+\infty$ for arbitrary $N$.
In view of~(\ref{ddx'0}), the integral in the last expression can be treated by the standard method of stationary phase
(providing that the support of $\chi(x)$ is contained in a sufficiently small neighborhood of the origin),
which gives the following asymptotic expression:
\begin{equation*}
  \left(\frac{2\pi}{s}\right)^{(n-1)/2} |\det \partial_{x'}^2 \Phi(0)|^{-1/2} e^{i\pi\, {\rm sgn}(\partial_{x'}^2 \Phi(0))/4} e^{i s \Phi(0)} F(0) 
  + O(s^{(n+1)/2}). 
\end{equation*}
Thus we arrive at~(\ref{stat-phase2}).
\end{proof}

\section{Asymptotic behavior of the function $u^f$}
Represent the integral in~(\ref{uf}) as the following sum
\begin{multline}
  \mathrm{v.p.}\int_{{\mathbb R}^d\times S^{n-1}\times (0,\infty)}
  \frac{e^{i(\langle{}x,\xi\rangle-\langle{}t, \sigma\rangle \rho\sqrt{\xi^2 + m^2})} \zeta(\rho) F(\xi,\sigma,\rho)}{1-\rho} \, d\xi dS_\sigma d\rho \\
  + \int_{{\mathbb R}^d\times S^{n-1}\times (0,\infty)}
  \frac{e^{i(\langle{}x,\xi\rangle-\langle{}t, \sigma\rangle \rho\sqrt{\xi^2 + m^2})} (1-\zeta(\rho))F(\xi,\sigma,\rho)}{1-\rho}  \, d\xi dS_\sigma d\rho,
  \label{sum-chi}
\end{multline}
where $\zeta(\rho)$ is a function from $C_0^\infty(0,+\infty)$ that is equal to unity in a neighborhood of $\rho=1$.
Substitution $\tau=\rho\sigma$ in the second term yields
\begin{equation*}
  \int_{{\mathbb R}^d\times{\mathbb R}^n}
  \frac{e^{i(\langle{}x,\xi\rangle-\langle{}t, \tau\rangle \sqrt{\xi^2 + m^2})} (\xi^2 + m^2)^{n/2-1} (1-\zeta(|\tau|))\hat f(\xi,\tau\sqrt{\xi^2 + m^2})}{1-\tau^2} \, d\xi d\tau.
\end{equation*}
For $x=s\theta$, $t=s\omega$,
the exponent of the exponential in the integrand takes the form
\begin{equation*}
  i s \Phi(\xi,\tau; \theta, \omega) = i s \left(\langle\theta,\xi\rangle - \langle\omega, \tau\rangle \sqrt{\xi^2 + m^2}\right).
\end{equation*}
Under the integral sign, this exponential is multiplied by an expression, which belongs to Schwartz space
if considered as a function of variables $(\xi,\tau)$.
Besides,
\begin{equation}
  |\partial_{\xi,\tau} \Phi| \geqslant |\partial_\tau \Phi| = \left|\omega \sqrt{\xi^2 + m^2}\right| \geqslant m,
  \label{ddtau}
\end{equation}
so Lemma~\ref{stat-triv} applies to the integral in consideration, which means that it is equal to $O(s^{-N})$ as $s\to\infty$
for arbitrary $N$.

Now turn to the first term in~(\ref{sum-chi}).
For $x=s\theta$, $t=s\omega$, it takes the following form
\begin{equation}
  \mathrm{v.p.}\int_{{\mathbb R}^d\times S^{n-1}\times (0,\infty)}
  \frac{1}{1-\rho}\, e^{i s \Phi(\xi,\sigma,\rho; \theta, \omega)}\, F_\zeta(\xi, \sigma, \rho) \,d\xi dS_\sigma d\rho,
  \label{vpint-sphere}
\end{equation}
where
\begin{gather}
  \Phi(\xi,\sigma,\rho; \theta, \omega) = \langle\theta,\xi\rangle - \langle\omega, \sigma\rangle\rho \sqrt{\xi^2 + m^2},   \label{Phi}\\
  F_\zeta(\xi, \sigma, \rho) = \zeta(\rho) F(\xi, \sigma, \rho). \notag
\end{gather}

Let $\gamma$ be a chart on the sphere $S^{n-1}$.
By $\gamma$ we will also denote the corresponding coordinates ranging an open subset of ${\mathbb R}^{n-1}$;
by $\partial_\gamma\Phi$ we denote the derivative
\begin{equation*}
  \partial_\gamma(\Phi(\xi,\sigma(\gamma),\rho; \theta, \omega)).
\end{equation*}
(In the case $n=1$, we assume that $\gamma=\sigma=\pm 1$.)
Let $\chi(\sigma)$ be a smooth function on the sphere $S^{n-1}$ supported in the domain of definition of the chart $\gamma$.
We pass from~(\ref{vpint-sphere}) to the integral of the same form for the function
\begin{equation*}
  F_{\chi\zeta}(\xi, \sigma, \rho) = \chi(\sigma) F_\zeta(\xi, \sigma, \rho).
\end{equation*}
This integral can be written as follows
\begin{equation}
  \mathrm{v.p.}\int_{{\mathbb R}^d\times{\mathbb R}^{n-1}\times{\mathbb R}} 
  \frac{1}{1-\rho}\, e^{i s \Phi(\xi,\sigma(\gamma),\rho; \theta, \omega)} F_{\chi\zeta}(\xi, \sigma(\gamma), \rho) J(\gamma) \,d\xi d\gamma d\rho,
  \label{vpint-sphere-local}
\end{equation}
where $J = \partial S_\sigma/\partial\gamma$ is the Jacobian determinant corresponding to the substitution $\sigma \mapsto \gamma(\sigma)$ 
(in the case $n=1$, the integral with respect to $\gamma$ equals the sum with respect to $\gamma=\pm 1$, and we also assume $J=1$, $\gamma(\sigma)=\sigma$). 
The set of integration here is larger than the domain of definition of the integrand.
This expression, however, makes sense as the integrand can be smoothly continued by zero to the set of integration.

In order to apply Lemma~\ref{stat-asymp}, we need to find points $(\xi,\sigma,\rho)$,
in which $\rho=1$, $\partial_\xi\Phi = 0$, and (in the case $n>1$) $\partial_\gamma\Phi = 0$.
For $n>1$ the last condition implies that $\sigma(\gamma) = \pm\omega$.
For $n=1$ the last equality holds automatically as $\omega$ and $\sigma$ take values $\pm 1$.
Next
\begin{equation}
  \partial_\xi\Phi = \theta - \frac{\langle\omega,\sigma\rangle \rho\, \xi}{\sqrt{\xi^2+m^2}} = \theta \mp \frac{\xi}{\sqrt{\xi^2+m^2}}.\label{dPhi}
\end{equation}
Hence the condition $\partial_\xi\Phi = 0$ means that
\begin{equation*}
  \xi^2 = \frac{m^2 \theta^2}{1-\theta^2},
  \quad \xi = \pm \frac{m\theta}{\sqrt{1-\theta^2}}.
\end{equation*}
Thus we find two critical points $(\xi,\sigma,\rho)$:
\begin{equation}
  (\xi, \sigma) = \pm \left(\frac{m\theta}{\sqrt{1-\theta^2}}, \omega\right), \quad \rho=1. \label{xis-crit}
\end{equation}

We can cover the sphere $S^{n-1}$ by a finite set of charts and choose corresponding functions $\chi(\sigma)$ in such a way that
their sum is identically unity on the sphere.
Then the integral~(\ref{vpint-sphere}) equals the sum of integrals of the form~(\ref{vpint-sphere-local}).
Consider first a chart $\gamma$, whose domain of definition does not contain the points $\pm\omega$ (this is necessary only in the case $n>1$).
We will check that the integral~(\ref{vpint-sphere-local}) equals $O(s^{-N})$ for any $N$ in this case.
In view of Lemma~\ref{stat-decay}, this is valid if $|\partial_{\xi,\gamma}\Phi| \geqslant c>0$ on the support of the integrand.
We have ($j\leqslant n-1$)
\begin{equation*}
  \partial_{\gamma_j}\Phi = -\langle\omega,\partial_{\gamma_j}\sigma\rangle \rho \sqrt{\xi^2+m^2}.
\end{equation*}
As $\sigma(\gamma) \ne \pm\omega$, we can pick $j$ such that $\langle\omega,\partial_{\gamma_j}\sigma\rangle \ne 0$,
since the hyperplane tangent to the sphere $S^{n-1}$ at the point $\sigma$ is not orthogonal to $\omega$.
Besides, on the support of the integrand, the variable $\rho$ is separated from zero,
hence $|\partial_{\gamma_j}\Phi|\geqslant c> 0$.
This means that Lemma~\ref{stat-decay} applies to the integral~(\ref{vpint-sphere-local}).

Thus, to describe the asymptotic behavior of the integral~(\ref{vpint-sphere}), it is sufficient to consider the integral~(\ref{vpint-sphere-local})
for charts $\gamma$, whose domain of definition contain $\pm\omega$.
Let the domain of definition of a chart $\gamma$ contain $\omega$, do not contain $-\omega$,
and, besides, let domains of definition of other charts of the atlas do not contain $\omega$. 
In particular, this implies that the corresponding function $\chi$ satisfies $\chi(\omega)=1$
(in the case $n=1$, when $\omega,\,\sigma=\pm 1$, we assume that $\chi(\omega)=1$, $\chi(-\omega)=0$).
Asymptotic properties of the integral~(\ref{vpint-sphere-local}) are determined by the behavior of the integrand
at the critical point $\varkappa = (\xi, \sigma, \rho)$ given by equality~(\ref{xis-crit}) with the sign $+$.
This is due to Lemma~\ref{stat-asymp}.
In order to apply the latter, however, we need to localize the integral~(\ref{vpint-sphere-local}) with respect to $\xi$.
To this end, we introduce one more smooth compactly supported cut-off function $\eta(\xi)$ that is equal to unity
in a neighborhood of the point
$\xi=m\theta/\sqrt{1-\theta^2}$ (see~(\ref{xis-crit}))
and decompose the integral~(\ref{vpint-sphere-local}) as follows
\begin{multline*}
  \mathrm{v.p.}\int_{{\mathbb R}^d\times{\mathbb R}^{n-1}\times{\mathbb R}} 
  \frac{-1}{\rho-1}\, e^{i s \Phi(\xi,\sigma(\gamma),\rho; \theta, \omega)}\, F_{\eta\chi\zeta}(\xi, \sigma(\gamma), \rho) J(\gamma) \,d\xi d\gamma d\rho \\
  + \mathrm{v.p.}\int_{{\mathbb R}^d\times{\mathbb R}^{n-1}\times{\mathbb R}} 
   \frac{-1}{\rho-1}\, e^{i s \Phi(\xi,\sigma(\gamma),\rho; \theta, \omega)}\, (1-\eta(\xi)) F_{\chi\zeta}(\xi, \sigma(\gamma), \rho) J(\gamma) \,d\xi d\gamma d\rho,
\end{multline*}
where
\begin{equation*}
  F_{\eta\chi\zeta}(\xi, \sigma, \rho) = \eta(\xi) F_{\chi\zeta}(\xi, \sigma, \rho).
\end{equation*}
To estimate the second term, we again apply Lemma~\ref{stat-decay}.
Turn to the hypothesis~(\ref{ddx'}) of this lemma.
As can be seen from the first equality in~(\ref{dPhi}),
if $\langle\omega,\sigma\rangle$ and $\rho$ are sufficiently close to unity, and $\xi$ is sufficiently large,
then we have $|\partial_\xi\Phi| \geqslant c > 0$, since $|\theta| < 1$.
To fulfill these conditions, we demand that:
the chart $\gamma$ on $S^{n-1}$ be determined in a sufficiently small neighborhood of the point $\omega$;
the function $\zeta(\rho)$ be nonzero only in a sufficiently small neighborhood of the point $\rho=1$;
the difference $1-\eta(\xi)$ be nonzero only for sufficiently large $\xi$.

We will treat the first term with the help of Lemma~\ref{stat-asymp}.
In the conditions of this lemma (more specifically, in~(\ref{ddx'ne0}), (\ref{ddx'0})), the critical point zero
should be substituted by $\varkappa$ in the present context.
Condition~(\ref{ddxne0}) reads $\partial_{\xi,\gamma,\rho}\Phi \ne 0$.
This relation follows, however, from~(\ref{ddtau}).
The first equality of condition~(\ref{ddx'0}) has already been checked.
Condition~(\ref{ddx'ne0}) says that the are no critical points other than $\varkappa$ in a neighborhood
of the support of the integrand.
This is valid in our situation, since the domain of definition of the chart $\gamma$ by assumption does not contain
$-\omega$, which corresponds to the second critical point in~(\ref{xis-crit}).

Finally, we turn to calculation of quantities occurring in~(\ref{stat-phase2}).
Choose the coordinate axes in such a way that 
\begin{equation*}
  \theta=|\theta| e_d\in{\mathbb R}^d, \quad \omega=e_n\in{\mathbb R}^n.
\end{equation*}
Then
\begin{equation*}
  \Phi(\xi,\sigma,\rho; \theta, \omega) = |\theta| \xi_d - \sigma_n\rho \sqrt{\xi^2 + m^2},
\end{equation*}
and the first equality in~(\ref{dPhi}) now reads
\begin{equation*}
  \partial_\xi\Phi = |\theta| e_d - \frac{\sigma_n \rho\xi}{\sqrt{\xi^2+m^2}}.
\end{equation*}
The critical point in consideration, which is $\varkappa$, is given by equalities
\begin{equation*}
  \xi = \frac{m|\theta| e_d}{\sqrt{1-\theta^2}}, \quad \sigma = e_n, \quad \rho=1.
\end{equation*}
We have
\begin{equation*}
  \Phi(\varkappa; \theta,\omega) = -m\sqrt{1-\theta^2}.
\end{equation*}
Next ($k,j \leqslant d$)
\begin{equation*}
  \partial_{\xi_k}\partial_{\xi_j} \Phi = \frac{-\sigma_n \rho \delta_k^j}{\sqrt{\xi^2+m^2}} + \frac{\sigma_n \rho \xi_k \xi_j}{(\xi^2+m^2)^{3/2}}.
\end{equation*}
At the point $\varkappa$, we have
\begin{multline*}
  \partial_{\xi_k}\partial_{\xi_j} \Phi(\varkappa; \theta,\omega) = \frac{-\delta_k^j}{\sqrt{\xi^2+m^2}} + \frac{\xi^2 \delta_k^d \delta_j^d}{(\xi^2+m^2)^{3/2}} =
  \frac{-\delta_k^j \sqrt{1-\theta^2}}{m} + \frac{\xi^2 \delta_k^d \delta_j^d}{(\xi^2+m^2)^{3/2}} \\
  =\frac{-\sqrt{1-\theta^2}}{m} (\delta_k^j - \theta^2 \delta_k^d \delta_j^d).
\end{multline*}
Next we calculate the derivatives of the function $\Phi$ with respect to $\gamma$.
This step is required only for $n>1$.
It is convenient to choose a chart $\gamma$ on the sphere in such a way that
\begin{equation}
  \sigma(\gamma) = (\gamma, \sqrt{1-\gamma^2}).
  \label{sg}
\end{equation}
We have ($k,j \leqslant n-1$)
\begin{gather*}
  \partial_{\gamma_j} \Phi = -\rho \sqrt{\xi^2 + m^2}\, \partial_{\gamma_j} \sigma_n =
  \rho \sqrt{\xi^2 + m^2}\, \frac{\gamma_j}{\sqrt{1-\gamma^2}}, \\
  \partial_{\gamma_k}\partial_{\gamma_j} \Phi = \rho \sqrt{\xi^2 + m^2} \left(\frac{\delta_k^j}{\sqrt{1-\gamma^2}} + \frac{\gamma_j \gamma_k}{(1-\gamma^2)^{3/2}}\right).
\end{gather*}
At the point $\varkappa$ we have $\gamma=0$, and so
\begin{equation*}
  \partial_{\gamma_k}\partial_{\gamma_j} \Phi(\varkappa; \theta,\omega) = \frac{m \delta_k^j}{\sqrt{1-\theta^2}}.
\end{equation*}
Next ($k\leqslant d$, $j\leqslant n-1$)
\begin{equation*}
  \partial_{\xi_k}\partial_{\gamma_j} \Phi =
  \frac{\rho \xi_k \gamma_j}{\sqrt{(\xi^2 + m^2)(1-\gamma^2)}},
  \quad
  \partial_{\xi_k}\partial_{\gamma_j} \Phi(\varkappa; \theta,\omega) = 0,
\end{equation*}
whence
\begin{equation*}
  \det \partial_{\xi,\gamma}^2 \Phi(\varkappa; \theta,\omega) = (\det \partial_\xi^2 \Phi \, \det \partial_\gamma^2 \Phi)(\varkappa; \theta,\omega)
\end{equation*}
(for $n=1$, we set by the definition $\partial_{\xi,\gamma}^2 \Phi = \partial_\xi^2 \Phi$,
$\det \partial_\gamma^2 \Phi=1$).
It follows from the equalities obtained above that
\begin{gather*}
  \det \partial_\xi^2 \Phi(\varkappa; \theta,\omega) = \left(\frac{-\sqrt{1-\theta^2}}{m}\right)^d (1-\theta^2), \quad
  \det \partial_\gamma^2 \Phi(\varkappa; \theta,\omega)  = \left(\frac{m}{\sqrt{1-\theta^2}}\right)^{n-1}.
\end{gather*}
We summarize our calculations by the following equalities
\begin{gather}
  \Phi(\varkappa; \theta,\omega) = -m\sqrt{1-\theta^2}, \quad
  |\det \partial_{\xi,\gamma}^2 \Phi(\varkappa; \theta,\omega)| = \frac{(1-\theta^2)^{(d-n+3)/2}}{m^{d-n+1}}, \notag\\
  {\rm sgn}(\partial_{\xi,\gamma}^2 \Phi(\varkappa; \theta,\omega)) = n-1-d.
  \label{Phi-deriv}
\end{gather}
We also need to find the sign of $\partial_\rho \Phi$ at $\varkappa$:
\begin{equation*}
  \partial_\rho \Phi(\varkappa; \theta,\omega) = -\sqrt{\xi^2+m^2} < 0.
\end{equation*}

Now find $J$ occurring in the integral~(\ref{vpint-sphere-local}).
We have ($j\leqslant n-1$)
\begin{equation*}
  \partial_j\sigma = e_j + e_n \partial_j \sigma_n = e_j - \frac{\gamma_j e_n}{\sqrt{1-\gamma^2}},
\end{equation*}
therefore, at the point $\varkappa$, in which $\gamma=0$, we have
\begin{equation*}
  J(0) = |\det \{\langle\partial_j\sigma(0), \partial_k\sigma(0)\rangle\}_{j,k}|^{1/2} = 1.
\end{equation*}
It remains to find the value of the function $F_{\eta\chi\zeta}$ at $\varkappa$: 
\begin{equation*}
  F_{\eta\chi\zeta}(\varkappa) = F(\varkappa) =
  \frac{m^{n-2}}{2(1-\theta^2)^{n/2-1}}\, \hat f\left(\frac{m\theta}{\sqrt{1-\theta^2}}, \frac{m\omega}{\sqrt{1-\theta^2}}\right)
\end{equation*}
(we have taken into account that $\eta=\chi=\zeta=1$ at this point).
Thus we derive the following asymptotic formula for the integral~(\ref{vpint-sphere-local}) 
as $s \to +\infty$:
\begin{equation*}
  i\pi\, \left(\frac{2\pi}{s}\right)^{(d+n-1)/2} \frac{m^{(d-n+1)/2}}{(1-\theta^2)^{(d-n+3)/4}} e^{i\pi\, (n-d-1)/4} e^{-i s\, m\sqrt{1-\theta^2}} F(\varkappa) \\
  + O(s^{(d+n+1)/2}).
\end{equation*}
The leading term here can be written as
\begin{equation*}
  \frac{(2\pi)^{d+n}}{s^{(d+n-1)/2}}\, e^{-i s\, m\sqrt{1-\theta^2}} U_-^f(\theta,\omega),
\end{equation*}
where
\begin{multline*}
  U_-^f(\theta, \omega) =
  i\pi\, (2\pi)^{-(d+n+1)/2} \frac{m^{(d-n+1)/2}}{(1-\theta^2)^{(d-n+3)/4}} e^{i\pi\, (n-d-1)/4} F(\varkappa) \\
  = \frac{e^{i\pi (n-d+1)/4}\, m^{(d+n-3)/2}}{4(2\pi)^{(d+n-1)/2} (1-\theta^2)^{(d+n-1)/4}}\, \hat f\left(\frac{m\theta}{\sqrt{1-\theta^2}}, \frac{m\omega}{\sqrt{1-\theta^2}}\right).
\end{multline*}

Consider the second critical point in~(\ref{xis-crit}), which will be denoted by $\varkappa'$.
Similarly to~(\ref{sg}), we choose coordinates $\gamma$ as follows
\begin{equation*}
  \sigma(\gamma) = (\gamma, -\sqrt{1-\gamma^2}).   
\end{equation*}
Then the calculations analogous to those performed previously yield the equalities
\begin{gather}
  \Phi(\varkappa'; \theta, \omega) = m\sqrt{1-\theta^2}, \quad
  |\det \partial_{\xi,\gamma}^2 \Phi(\varkappa'; \theta, \omega)| = \frac{(1-\theta^2)^{(d-n+3)/2}}{m^{d-n+1}}, \notag\\
  {\rm sgn}(\partial_{\xi,\gamma}^2 \Phi(\varkappa'; \theta, \omega)) = d-n+1
  \label{Phi-deriv2}
\end{gather}
and, besides,
\begin{equation*}
  \partial_\rho \Phi(\varkappa'; \theta, \omega) = \sqrt{\xi^2+m^2} > 0.
\end{equation*}
So the leading term in the asymptotic expansion of the integral~(\ref{vpint-sphere-local}) can be written as
\begin{equation*}
  \frac{(2\pi)^{d+n}}{s^{(d+n-1)/2}}\, e^{i s\, m\sqrt{1-\theta^2}} U_+^f(\theta,\omega),
\end{equation*}
where
\begin{equation*}
  U_+^f(\theta, \omega)
  = \frac{e^{i\pi (d-n-1)/4}\, m^{(d+n-3)/2}}{4(2\pi)^{(d+n-1)/2} (1-\theta^2)^{(d+n-1)/4}}\, \hat f\left(\frac{-m\theta}{\sqrt{1-\theta^2}}, \frac{-m\omega}{\sqrt{1-\theta^2}}\right).
\end{equation*}
We can rewrite the expressions for $U^f_\pm$ in the following form
\begin{equation*}
  U_\pm^f(\theta, \omega)
  = \frac{e^{\pm i\pi (d-n-1)/4}}{4 m} \left(\frac{m}{2\pi\sqrt{1-\theta^2}}\right)^{(d+n-1)/2}\, \hat f\left(\frac{\mp m\theta}{\sqrt{1-\theta^2}}, \frac{\mp m\omega}{\sqrt{1-\theta^2}}\right).
\end{equation*}
Relation~(\ref{apm-f}) is a consequence of this formula.

\section{Asymptotic behavior of the function $u^a$} 
For $x(s) = s\theta$, $t(s) = s\omega$, we have
\begin{gather}
  u^a(s\theta, s\omega) = (2\pi)^{-d-n} \int_{{\mathbb R}^d\times S^{n-1}} e^{i s\Phi(\xi,\sigma; \theta, \omega)} A(\xi,\sigma) d\xi dS_\sigma, \notag \\
  \Phi(\xi,\sigma; \theta, \omega) = \langle\theta,\xi\rangle - \langle\omega, \sigma\rangle \sqrt{\xi^2 + m^2}. \label{PhiA}
\end{gather}
To describe asymptotic properties of this integral as $s\to+\infty$, we apply the method of stationary phase.
The phase function $\Phi$ coincides with the function defined in~(\ref{Phi}) restricted to the set $\{\rho=1\}$.
This allows us to use the results of calculations from the previous section.
In particular, the critical points $(\xi,\sigma)$, in which $\partial_{\xi,\sigma}\Phi=0$,
are given by the first equality in~(\ref{xis-crit}).
The values of the function $\Phi$ and its derivatives at critical points, which are required for asymptotic analysis, 
are given by formulas~(\ref{Phi-deriv}), (\ref{Phi-deriv2}).
The justification of the method of stationary phase (localization with respect to $\xi$, passing to local coordinates on the sphere $S^{n-1}$)
is quite analogous to that made in the previous section,
so we will not expose it here.

Thus the function $u^a$ has the asymptotic property of the form~(\ref{ampl}) with the following coefficients 
\begin{equation*}
  U^a_\pm(\theta, \omega)
  = \frac{e^{\pm i\pi (d-n+1)/4}}{(2\pi)^{(d+n+1)/2}} \frac{m^{(d-n+1)/2}}{(1-\theta^2)^{(d-n+3)/4}}\, A\left(\frac{\mp m\theta}{\sqrt{1-\theta^2}}, \mp\omega\right).
\end{equation*}
This relation can be written in terms of the density $a(\xi,\tau)$ as follows
\begin{equation*}
  U^a_\pm(\theta, \omega)
  = \frac{e^{\pm i\pi (d-n+1)/4}}{4\pi m} \left(\frac{m}{2\pi \sqrt{1-\theta^2}}\right)^{(d+n-1)/2}\,
  a\left(\frac{\mp m\theta}{\sqrt{1-\theta^2}}, \frac{\mp m\omega}{\sqrt{1-\theta^2}}\right).
\end{equation*}
Relation~(\ref{apm}) follows from this formula.

The formulas for $U^f_\pm$, $U^a_\pm$ obtained above prove Theorem~\ref{thm-asymp}.
Formula~(\ref{Upm}) allows expressing the density $a$ via the function $\hat f$ and one of the functions $U_\pm$.
This leads to the relation given in the formulation of Theorem~\ref{thm-exist}.
Inclusion $a\in \mathcal{S}(\Sigma_m)$ follows from this relation.
The density $a$ is determined uniquely at every point of $\Sigma_m$, which leads to the uniqueness of solution of the form~(\ref{uint})
that has the asymptotic property~(\ref{ampl}) with a given coefficient $U_+$ (or $U_-$).
Thus Theorem~\ref{thm-exist} is proved.

\section{Asymptotic behavior of the function $u$ along chara\-cte\-ris\-tic directions}
Now we discuss the behavior of solution~(\ref{uint}) when $(x,t)$ tends to infinity along a chara\-cte\-ris\-tic direction,
which will be parameterized by unit vectors $(\theta,\omega)\in S^{d-1} \times S^{n-1}$ and a parameter $q\in{\mathbb R}$:
\begin{equation}
  x(s) = (s+q)\theta, \quad t(s) = s\omega.
  \label{xsts-char}
\end{equation}
For simplicity, we will assume $f=0$.
Then $u=u^a$, and we need to examine the behavior of the integral
\begin{equation*}
  u(x(s), t(s)) = (2\pi)^{-d-n} \int_{{\mathbb R}^d\times S^{n-1}} e^{i s\Phi(\xi,\sigma; \theta, \omega)} e^{i q \langle\theta,\xi\rangle} a(\xi,\sigma) \, d\xi dS_\sigma
\end{equation*}
as $s\to\infty$, where the function $\Phi$ is given by~(\ref{PhiA}).
The product $e^{i q \langle\theta,\xi\rangle} a(\xi,\sigma)$ considered as a function of variables $(\xi,\sigma)$ belongs to Schwartz space.
Next
\begin{equation*}
  |\partial_\xi\Phi| = \left|\theta - \frac{\langle\omega,\sigma\rangle \xi}{\sqrt{\xi^2+m^2}}\right| \geqslant 1 - \frac{|\xi|}{\sqrt{\xi^2+m^2}}
  \geqslant \frac{C_m}{1+|\xi|}.
\end{equation*}
This means that Lemma~\ref{stat-triv} applies to the integral in the preceding formula
(after sub\-sti\-tu\-ting $\sigma$ by local coordinates), since condition~(\ref{ddPhi}) is satisfied for $M=-1$.
Thus for any $N$ we have
\begin{equation*}
  u(x(s), t(s)) = O(s^{-N}), \quad s\to\infty.
\end{equation*}

A particular case of equation~(\ref{eqn}) is Klein-Gordon-Fock equation ($n=1$, $f=0$), for which
the Cauchy problem with the initial data given, for example, on the hypersurface $\{t=0\}$:
\begin{equation*}
  u|_{t=0} = u_0, \quad u|_{t=0} = u_1,
\end{equation*}
is well-posed.
If the initial data $u_0$, $u_1$ belong to $\mathcal{S}({\mathbb R}^d)$, then
the solution to this problem can be written in the form~(\ref{uaint}) with the following density $A(\xi,\sigma)$ (in the present context, $\sigma=\pm 1$)
\begin{equation*}
  A(\xi, \pm 1) = \frac{\hat u_0(\xi)}{2} \pm \frac{i \hat u_1(\xi)}{2\sqrt{\xi^2+m^2}}.
\end{equation*}
(In the case under consideration, when $n=1$, formula~(\ref{uaint}) with such a function $A$ can,
of course, be obtained by the standard Fourier method.)
Thus our results apply to a solution to the Cauchy problem.
In particular, the last equality together with formula~(\ref{Upm}) imply that the asymptotic behavior
of the solution along a timelike direction is described in terms of the Fourier transform of the initial data.
However, the solution decays faster than any power of distance and time along 
a characteristic direction. 
Thus we see a distinct contrast between behavior of solutions to Klein-Gordon-Fock equation and that of the wave equation.
Solutions to the analogous Cauchy problem for the wave equation decay rapidly along timelike directions
(for odd $d$, this follows from Huygens principle), and exhibit nontrivial asymptotic behavior along characteristic directions.
In the case $d=3$, the asymptotic behavior is described by the following relation~\cite{Blag-scatter}
\begin{equation*}
  u(x(s), t(s)) = \frac{1}{4\pi s} \left(-\partial_q (R u_0(\theta,q)) + R u_1(\theta,q)\right) + o(s^{-1}), \quad s\to \pm\infty,
\end{equation*}
in which $x(s)$, $t(s)$ are given by~(\ref{xsts-char}), $R$ is Radon transform:
\begin{equation*}
  R v(\theta,q) = \int_{\langle{}x,\theta\rangle=q} v(x) \,dS_x.
\end{equation*}
Thus in the case of the wave equation, the behavior of a solution at infinity is described in terms
of Radon transform of the initial data (rather than Fourier transform).


\begin{thebibliography}{}

\bibitem{Lax-Phillips}
  Lax P. D., Phillips R. S. Scattering Theory. New York–London: Academic Press, 1967.
  
\bibitem{Blag-scatter}
  A.~S.~Blagoveshchensky, On Some New Well-Posed Problems for the Wave Equation,
  in Proceedings of the V All-Union Symposium on Diffraction and Wave Propagation (1970), 29--35, Leningrad, Nauka, 1971 [in Russian].

\bibitem{Moses}
  H. E. Moses, R. T. Prosser, Acoustic and Electromagnetic Bullets: Derivation of New
Exact Solutions of the Acoustic and Maxwell’s Equations, SIAM J. Appl. Math.,
50:5 (1990), 1325--1340.

\bibitem{Kis}
  A.~P.~Kiselev, Localized Light Waves: Paraxial and Exact Solutions of the Wave Equation (a Review),
  Optics and Spectroscopy, 102:4 (2007), 603--622.

\bibitem{BV}
  M. I. Belishev, A. F. Vakulenko, On a Control Problem for the Wave Equation in ${\mathbb R}^3$,
  Journal of Mathematical Sciences, 142 (2007), 2528--2539.
  
\bibitem{Plachenov}
  A.~B.~Plachenov, Acoustic, Electromagnetic and Elastic Wavefield Energy
Expression via its Asymptotics at Large Times and Distances,
  Zap. Nauchn. Semin. POMI, 493 (2020), 269--287 [in Russian].

\bibitem{Blag-char}
A.~S.~Blagoveshchensky, On the Characteristic Problem for the Ultrahyperbolic Equation, Matem. Sbornik,
63:105 (1964), No.~1, 137--168 [in Russian].

\bibitem{deRham}
  Georges de Rham, Solution élémentaire d’opérateurs différentiels du second ordre,
  Annales de l’institut Fourier, 8 (1958), 337--366.

\bibitem{Ortner-Wagner}
  N. Ortner, P. Wagner, Fourier transformation of O(p, q)-invariant distributions.
  Fundamental solutions of ultra-hyperbolic operators,
  Journal of Mathe\-ma\-ti\-cal Analysis and Applications, 450 (2017), 262--292.


\end{thebibliography}
\end{document}